\documentclass[12pt, reqno]{amsart}
\makeatletter
\@namedef{subjclassname@1991}{$\mathrm{1991}$ Mathematics Subject Classification}
\@namedef{subjclassname@2000}{$\mathrm{2000}$ Mathematics Subject Classification}
\@namedef{subjclassname@2010}{$\mathrm{2010}$ Mathematics Subject Classification}
\@namedef{subjclassname@2020}{$\mathrm{2020}$ Mathematics Subject Classification}
\makeatother
\usepackage{amsmath,amsthm, amscd, amsfonts, amssymb, graphicx, color}
\usepackage[bookmarksnumbered, colorlinks, plainpages,linkcolor=blue,urlcolor=blue,citecolor=blue]{hyperref}
\newtheorem{thm}{Theorem}[section]
\newtheorem{cor}[thm]{Corollary}
\newtheorem{lem}[thm]{Lemma}

\numberwithin{equation}{section}

\begin{document}

\title{The sum of EP elements in a ring with involution}

\author{Huanyin Chen}
\author{Marjan Sheibani$^*$}
\address{
School of Mathematics\\ Hangzhou Normal University\\ Hangzhou, China}
\email{<huanyinchenhz@163.com>}
\address{Farzanegan Campus, Semnan University, Semnan, Iran}
\email{<m.sheibani@semnan.ac.ir>}


\begin{abstract}  We present a necessary and sufficient condition under which
the sum of two EP elements in a *-ring has core inverse. As an application, we establish the conditions under which a block complex with EP subblocks has core inverse.
\end{abstract}

\thanks{Corresponding author: Marjan Sheibani}

\maketitle

\section{Introduction}

Let $R$ be an associative ring with an involution $*$. Here, $*: R\rightarrow R$ is an operation satisfying $(x+y)^*=x^*+y^*, (xy)^*=y^*x^*$ and $(x^*)^*=x$ for
all $x, y\in R$. A ring $R$ with involution $*$ is called a *-ring.

The core inverse of a complex matrix was first introduced by Baksalary and G. Trenkler in ~\cite{B}. Raki\'c et al. extended the core inverse of a complex to the case of an element in a *-ring (see~\cite{R}). An element $a\in R$ has core inverse if there exists some $x\in R$ such that $$axa=a, xR=aR, Rx=Ra^*.$$
If such $x$ exists, it is unique, and denote it by $a^{\tiny\textcircled{\#}}$. In 2019, Xu obtained a new characterization of core inverse (~see\cite[Theorem 3.1]{XCZ}). They proved that $a\in R$ has core inverse if and only if there exists some $x\in R$ such that  $$xa^2=a, ax^2=x, (ax)^*=ax.$$ In this case, $a^{\tiny\textcircled{\#}}=x$.

An element $a$ in $R$ has group inverse provided that there exists $x\in R$ such that $$xa^2=a, ax^2=x, ax=xa.$$ Such $x$ is unique if exists, denoted by $a^{\#}$, and called the group inverse of $a$. An element $a$ in $R$ has Drazin inverse provided that there exists $x\in R$ such that $$xa^{n+1}=a^n, ax^2=x, ax=xa.$$ Such $x$ is unique if exists, denoted by $a^D$, and called the Drazin inverse of $a$. Evidently, $a\in R$ has Drazin inverse if and only $a^n\in R$ has group inverse for some
natural number $n$. An element $a\in R$ has $(1,3)$-inverse provided that there exists some $x\in R$ such that $a=axa$ and $(ax)^*=ax$. We denote $x$ by $a^{(1,3)}$. We list several characterizations of core inverse in a *-ring.

\begin{thm} (see~\cite[Theorem 3.4]{LC}, \cite[Theorem 2.14]{R} and ~\cite[Theorem 2.6]{XCZ}). Let $R$ be a *-ring, and let $a\in R$. Then the following are equivalent:\end{thm}
\begin{enumerate}
\item [(1)] $a$ has core inverse.
\vspace{-.5mm}
\item [(2)] There exists $x\in R$ such that $axa=a, x=xax, xa^2=a, ax^2=x, (ax)^*=ax$.
\vspace{-.5mm}
\item [(4)] There exists some $p^2=p=p^*\in R$ (i.e., $p$ is a projection) such that $pa=0$ and $a+p\in R$ is invertible.
\vspace{-.5mm}
\item [(5)] $a\in R$ has group inverse and $Ra=Ra^*a$.
\vspace{-.5mm}
\item [(6)] $a\in R$ has group inverse and $a\in R$ has $(1,3)$-inverse.\\
In this case, $a^{\tiny\textcircled{\#}}=x$.
\end{enumerate}

The core invertibility in a *-ring is attractive. Many authors have studied such problems from many different views, e.g., ~\cite{BT,K,K2,XS,XCZ,Z2}.

In~\cite[Theorem 4.3]{XCZ}, Xue, Chen and Zhang proved that $a+b\in R$ has core inverse under the conditions $ab=0$ and $a^*b=0$ for two core invertible elements $a$ and $b$ in $R$.

In ~\cite[Theorem 4.1]{ZCX}, Zhou et al. considered the core inverse of $a+b$ under the conditions $a^2a^{\tiny\textcircled{\#}}b^{\tiny\textcircled{\#}}b=baa^{\tiny\textcircled{\#}}, ab^{\tiny\textcircled{\#}}b=aa^{\tiny\textcircled{\#}}b$ in a Dedekind-finite ring in which $2$ is invertible.

Recall that $a\in R$ is EP (i.e., an EP element) provided that $a$ has group inverse $a^{\#}$ and Moore-Penrose inverse $a^{\dag}$ with $a^{\#}=a^{\dag}$(see~\cite{XCB}). In 2019, Xu et al. proved that $a\in R$ is EP if and only if there exists some $x\in R$ such that $$xa^2=a, ax^2=x, (xa)^*=xa.$$
In ~\cite{R}, Raki\'c et al. proved that $a\in R$ is EP if and only if $a\in R$ has core inverse and $a^{\tiny\textcircled{\#}}=a^{\#}$. For more results on EP elements in a ring, we refer ~\cite{W,XCB,Z2} to the reader.

In this paper, we present a new additive result for the core inverse in a *-ring. We give a necessary and sufficient condition under which
the sum of two EP elements has core inverse.

Let $C^{n\times n}$ be a *-ring of $n\times n$ complex matrices, with conjugate transpose as the involution. As an application, we establish perturbed conditions under which a block complex matrix with EP subblocks has core inverse.

Throughout the paper, all *-rings are associative with an identity. Let $a\in R^D$ and $a^{\pi}=1-aa^D$.
Let $p^2=p\in R$, and let $x\in R$. We write $x=pxp+px(1-p)+(1-p)xp+(1-p)x(1-p),$ and induce a Pierce representation given by the matrix
$x=\left(\begin{array}{cc}
pxp&px(1-p)\\
(1-p)xp&(1-p)x(1-p)
\end{array}
\right)_p.$ We use $R^{\#}$ and $R^{\tiny\textcircled{\#}}$ to
denote the sets of all group and core invertible elements in $R$, respectively. ${\Bbb C}^*$ stands for the set of all nonzero complex number and $A^*$ denotes the conjugate transpose $\overline{A}^T$ of the complex matrix $A$.

\section{Key lemmas}

This section is to investigate elementary properties of the core inverse in a *-ring which will be used in the sequel. We begin with

\begin{lem} Let $p\in R$ be an idempotent, $a\in R^{\#}$ and $pap^{\pi}=0$. If $ap^{\pi}\in R^{\#}$, then $(ap^{\pi})(ap^{\pi})^{\#}=(aa^{\#})p^{\pi}.$\end{lem}
\begin{proof} Since $pap^{\pi}=0$, we see that $$\begin{array}{lll}
ap^{\pi}-(a^{\#}p^{\pi})(ap^{\pi})^2&=&ap^{\pi}-(a^{\#}a^2)p^{\pi}\\
&=&(a-a^{\#}a^2)p^{\pi}=0.
\end{array}$$ Hence, $ap^{\pi}=(a^{\#}p^{\pi})(ap^{\pi})^2$. It follows that
$$\begin{array}{lll}
(ap^{\pi})(ap^{\pi})^{\#}&=&(a^{\#}p^{\pi})(ap^{\pi})^2(ap^{\pi})^{\#}\\
&=&(a^{\#}p^{\pi})(ap^{\pi})=(aa^{\#})p^{\pi},
\end{array}$$ as required.\end{proof}

\begin{lem} Let $p\in R$ be a projection, $x\in R$ and $x=\left(
\begin{array}{cc}
a&0\\
b&d
\end{array}
\right)_p$. Assume that $a\in R^{\#}$. Then $x\in R^{\#}$ if and only if $d\in R^{\#}$ and $d^{\pi}ba^{\pi}=0$. In this case,
$$x^{\#}=\left(
\begin{array}{cc}
a^{\#}&0\\
z&d^{\#}
\end{array}
\right)_p,$$ where
$$z=(d^{\#})^2ba^{\pi}+d^{\pi}b(a^{\#})^2-d^{\#}ba^{\#}.$$
\end{lem}
\begin{proof} This is the dual of ~\cite[Theorem 2.1]{MD}. We easily obtained it by applying ~\cite[Theorem 2.1]{MD} to the opposite ring $M_2(R^{op})$.\end{proof}

\begin{lem} Let $p\in R$ be a projection, $a\in R$ and $pap^{\pi}=0$. If $pap, ap^{\pi}\in R^{\tiny\textcircled{\#}}$ and $(ap^{\pi})^{\pi}p^{\pi}ap=0,$ then $a\in R^{\tiny\textcircled{\#}}$ and $pa^{\tiny\textcircled{\#}}p^{\pi}=0$.\end{lem}
\begin{proof}  Since $pap^{\pi}=0$, we have $a=\left(\begin{array}{cc}
pap&0\\
p^{\pi}ap&ap^{\pi}
\end{array}
\right)_p.$ By virtue of Theorem 1.1, $pap, ap^{\pi}\in R^{\#}$. As $(ap^{\pi})^{\pi}p^{\pi}apap=0,$ it follows by Lemma 2.2 that
$a\in R^{\#}$. Moreover, we have
$$a^{\#}=\left(
\begin{array}{cc}
(pap)^{\#}&0\\
z&(ap^{\pi})^{\#}
\end{array}
\right)_p,$$ where $$z=[(ap^{\pi})^{\#}]^2p^{\pi}ap(pap)^{\pi}+(ap^{\pi})^{\pi}p^{\pi}ap[(pap)^{\#}]^2-(ap^{\pi})^{\#}p^{\pi}ap(pap)^{\#}.$$
Since $p\in R$ is a projection, we have $$(pap)^{\tiny\textcircled{\#}}=[p\cdot (pap)\cdot p]^{\tiny\textcircled{\#}}=p(pap)^{\tiny\textcircled{\#}}p\subseteq pRp.$$ Similarly, $$(ap^{\pi})^{\tiny\textcircled{\#}}=(p^{\pi}ap^{\pi})^{\tiny\textcircled{\#}}\subseteq p^{\pi}Rp^{\pi}.$$
Set $$x=\left(
\begin{array}{cc}
(pap)^{\tiny\textcircled{\#}}&0\\
-(ap^{\pi})^{\tiny\textcircled{\#}}(p^{\pi}ap)(pap)^{\tiny\textcircled{\#}}&(ap^{\pi})^{\tiny\textcircled{\#}}
\end{array}
\right)_p.$$ Then we have
$$\begin{array}{rll}
ax&=&\left(\begin{array}{cc}
pap&0\\
p^{\pi}ap&ap^{\pi}
\end{array}
\right)_p\left(
\begin{array}{cc}
(pap)^{\tiny\textcircled{\#}}&0\\
-(ap^{\pi})^{\tiny\textcircled{\#}}(p^{\pi}ap)(pap)^{\tiny\textcircled{\#}}&(ap^{\pi})^{\tiny\textcircled{\#}}
\end{array}
\right)_p\\
&=&\left(
\begin{array}{cc}
(pap)(pap)^{\tiny\textcircled{\#}}&0\\
(ap^{\pi})^{\pi}p^{\pi}ap(pap)^{\tiny\textcircled{\#}}&(ap^{\pi})(ap^{\pi})^{\tiny\textcircled{\#}}
\end{array}
\right)_p\\
&=&\left(
\begin{array}{cc}
(pap)(pap)^{\tiny\textcircled{\#}}&0\\
0&(ap^{\pi})(ap^{\pi})^{\tiny\textcircled{\#}}
\end{array}
\right)_p.
\end{array}$$ Hence $$(ax)^*=\left(
\begin{array}{cc}
(pap)(pap)^{\tiny\textcircled{\#}}&0\\
0&(ap^{\pi})(ap^{\pi})^{\tiny\textcircled{\#}}
\end{array}
\right)_p^*=ax.$$
We further verify that $$\begin{array}{rll}
axa&=&\left(
\begin{array}{cc}
(pap)(pap)^{\tiny\textcircled{\#}}&0\\
0&(ap^{\pi})(ap^{\pi})^{\tiny\textcircled{\#}}
\end{array}
\right)_p\left(\begin{array}{cc}
pap&0\\
p^{\pi}ap&ap^{\pi}
\end{array}
\right)_p\\
&=&a,
\end{array}$$ and so $a\in R^{(1,3)}$. According to Theorem 1.1, $a$ has core inverse.

Moreover, we have $$\begin{array}{rl}
&a^{\tiny\textcircled{\#}}=a^{\#}ax=\left(
\begin{array}{cc}
(pap)^{\#}&0\\
z&(ap^{\pi})^{\#}
\end{array}
\right)_p\left(
\begin{array}{cc}
(pap)(pap)^{\tiny\textcircled{\#}}&0\\
0&(ap^{\pi})(ap^{\pi})^{\tiny\textcircled{\#}}
\end{array}
\right)_p\\
=&\left(
\begin{array}{cc}
(pap)^{\#}&0\\
z(pap)(pap)^{\tiny\textcircled{\#}}&(ap^{\pi})^{\#}
\end{array}
\right)_p.
\end{array}$$ Therefore $pa^{\tiny\textcircled{\#}}p^{\pi}=0$, as asserted.\end{proof}

\begin{lem} Let $a\in R^{EP}, b\in R^{\tiny\textcircled{\#}}$. If $a^{\pi}b\in R^D$ and $aba^{\pi}=0$, then the following are equivalent:\end{lem}
\begin{enumerate}
\item [(1)] $a+b\in R^{\tiny\textcircled{\#}}$ and $a(a+b)^{\tiny\textcircled{\#}}a^{\pi}=0$.
\vspace{-.5mm}
\item [(2)] $a(1+a^{\#}b)\in R^{\tiny\textcircled{\#}}$ and $b^{\pi}a^{\pi}b=0$.
\end{enumerate}
\begin{proof} Let $p=aa^{\#}$. Since $a\in R^{EP}$, we see that $p=aa^{\tiny\textcircled{\#}}$ is a projection.
Clearly, $pbp^{\pi}=a^{\#}(aba^{\pi})=0$. Then we have $$a=\left(
\begin{array}{cc}
a_1&0\\
0&0
\end{array}
\right)_p, b=\left(
\begin{array}{cc}
b_1&0\\
b_3&b_4
\end{array}
\right)_p.$$
Thus, $$a+b=\left(
\begin{array}{cc}
a_1+b_1&0\\
b_3&b_4
\end{array}
\right)_p.$$ Here, $$\begin{array}{rll}
a_1+b_1&=&aa^{\#}(a+b)aa^{\#}=aa^{\#}(a+b),\\
b_4&=&a^{\pi}(a+b)a^{\pi}=ba^{\pi}.
\end{array}$$
In view of Lemma 2.1, we directly compute that $$\begin{array}{rll}
b_4^{\#}&=&(ba^{\pi})^{\#}=b^{\#}a^{\pi},\\
b_4^{\pi}&=&a^{\pi}-b^{\#}a^{\pi}ba^{\pi}=a^{\pi}-b^{\#}ba^{\pi}=b^{\pi}a^{\pi}.
\end{array}$$

$(1)\Rightarrow (2)$ Set $x=a+b$ and $w=a_1+b_1$. Since $aa^{\#}(a+b)^{\tiny\textcircled{\#}}a^{\pi}=0$, we may write
$x^{\tiny\textcircled{\#}}=\left(
\begin{array}{cc}
x_1&0\\
x_3&x_4
\end{array}
\right)_p.$ In light of Theorem 1.1, we have
$$(xx^{\tiny\textcircled{\#}})^*=xx^{\tiny\textcircled{\#}}, x(x^{\tiny\textcircled{\#}})^2=x^{\tiny\textcircled{\#}}, x^{\tiny\textcircled{\#}}x^2=x~\mbox{and}~xx^{\tiny\textcircled{\#}}x=x.$$
Hence, $(wx_1)^*=wx_1, wx_1^2=x_1, x_1w^2=w.$ Therefore $w=aa^{\#}(a+b)=a(1+a^{\#}b)\in R^{\tiny\textcircled{\#}}$.

Since $xx^{\tiny\textcircled{\#}}x=x$, we have $$\begin{array}{rll}
\left(
\begin{array}{cc}
w&0\\
b_3&b_4
\end{array}
\right)_p&=&\left(
\begin{array}{cc}
w&0\\
b_3&b_4
\end{array}
\right)_p\left(
\begin{array}{cc}
x_1&0\\
x_3&x_4
\end{array}
\right)_p\left(
\begin{array}{cc}
w&0\\
b_3&b_4
\end{array}
\right)_p\\
&=&\left(
\begin{array}{cc}
xx_1w&0\\
(b_3x_1+b_4x_3)x+b_4x_4b_3&b_4x_4b_4
\end{array}
\right)_p.
\end{array}$$ Furthermore, we have
$$\begin{array}{rll}
\left(
\begin{array}{cc}
xx_1&0\\
b_3x_1+b_4x_3&b_4x_4
\end{array}
\right)_p&=&\left(
\begin{array}{cc}
w&0\\
b_3&b_4
\end{array}
\right)_p\left(
\begin{array}{cc}
x_1&0\\
x_3&x_4
\end{array}
\right)_p\\
&=&xx^{\tiny\textcircled{\#}}=(xx^{\tiny\textcircled{\#}})^*\\
&=&\left(\begin{array}{cc}
(xx_1)^*&(b_3x_1+b_4x_3)^*\\
0&(b_4x_4)^*
\end{array}
\right)_p.
\end{array}$$
Thus $$b_3=(b_3x_1+b_4x_3)w+b_4x_4b_3, b_3x_1+b_4x_3=0.$$
We infer that $b_3=b_4x_4b_3$, and so $(p^{\pi}-bb^{\#}p^{\pi})baa^{\#}=0$. This implies that
$b^{\pi}a^{\pi}baa^{\#}=0$. On the other hand, $b^{\pi}a^{\pi}ba^{\pi}=b^{\pi}ba^{\pi}-b^{\pi}a^{\#}(aba^{\pi})=0$. Therefore $b^{\pi}a^{\pi}b=b^{\pi}a^{\pi}b(aa^{\#}+a^{\pi})=0$.

$(2)\Rightarrow (1)$ We verify that
$$\begin{array}{lll}
p^{\pi}[1-b_4^{\#}b_4]b_3&=&a^{\pi}[1-b_4^{\#}b_4]a^{\pi}baa^{\#}\\
&=&a^{\pi}[1-(ba^{\pi})(ba^{\pi})^{\#}]a^{\pi}baa^{\#}\\
&=&a^{\pi}[1-bb^{\#}a^{\pi}]a^{\pi}baa^{\#}\\
&=&a^{\pi}b^{\pi}a^{\pi}baa^{\#}=0.
\end{array}$$
Hence,
$$\begin{array}{rll}
p(a+b)p^{\pi}&=&a^{\#}(aba^{\pi})=0,\\
p(a+b)p&=&a+aa^{\#}baa^{\#}=a(1+a^{\#}b)\in R^{\tiny\textcircled{\#}},\\
(a+b)p^{\pi}&=&ba^{\pi}\in R^{\tiny\textcircled{\#}}.
\end{array}$$
Since every core invertible element has group inverse, by Lemma 2.2, $a+b\in R^{\#}$. Moreover, we derive
$$\begin{array}{rll}
((a+b)p^{\pi})^{\pi}p^{\pi}(a+b)p&=&(1-(a+b)p^{\pi}((a+b)p^{\pi})^{\#})p^{\pi}(a+b)p\\
&=&(1-p^{\pi}(a+b)(a+b)^{\#})p^{\pi}(a+b)p\\
&=&p^{\pi}(1-(a+b)(a+b)^{\#})p^{\pi}(a+b)p\\
&=&p^{\pi}(1-b_4^{\#}b_4)p^{\pi}(a+b)p=0.
\end{array}$$ According to Lemma 2.3, $a+b\in R^{\tiny\textcircled{\#}}$ and $a(a+b)^{\tiny\textcircled{\#}}a^{\pi}=a^{\#}p(a+b)^{\tiny\textcircled{\#}}p^{\pi}=0$.\end{proof}

\section{The main result}

We have at our disposal all the information necessary to prove our main result.

\begin{thm} Let $a,b\in R^{EP}$ and $a^{\pi}b, b^{\pi}a\in R^D$. If $aba^{\pi}=bab^{\pi}=0$, then the following are equivalent:\end{thm}
\begin{enumerate}
\item [(1)] $a+b\in R^{\tiny\textcircled{\#}}, a(a+b)^{\tiny\textcircled{\#}}a^{\pi}=b(a+b)^{\tiny\textcircled{\#}}b^{\pi}=0$.
\vspace{-.5mm}
\item [(2)] $aa^{\#}b+bb^{\#}a\in R^{\tiny\textcircled{\#}}$ and $a^{\pi}b^{\pi}a=b^{\pi}a^{\pi}b=0$.
\end{enumerate}
\begin{proof}  $(1)\Rightarrow (2)$ In view of Lemma 2.4, $a(1+a^{\#}b)\in R^{\tiny\textcircled{\#}}$ and $b^{\pi}a^{\pi}b=0$. Analogously,
$a^{\pi}b^{\pi}a=0$. As in the proof in Lemma 2.4, we see that
$(a(1+a^{\#}b))^{\tiny\textcircled{\#}}=aa^{\#}(a+b)^{\tiny\textcircled{\#}}aa^{\#}$.
Let $q=aa^{\#}bb^{\#}$. Since $aba^{\pi}=0$, we have $b=\left(
\begin{array}{cc}
aa^{\#}b&0\\
a^{\pi}baa^{\#}&ba^{\pi}
\end{array}
\right)_{aa^{\#}}$. In view of ~\cite[Lemma 2.1]{XS}, $ba^{\pi}\in R^{\#}$. Then
$aa^{\#}b\in R^{\#}$ and $(aa^{\#}b)^{\#}=aa^{\#}b^{\#}$ by ~\cite[Theorem 2.3]{MD}. Hence $q^2=q=(aa^{\#}b)(aa^{\#}b)^{\#}$.
Thus $$\begin{array}{rll}
q(a(1+a^{\#}b))(1-q)&=&aa^{\#}bb^{\#}ab^{\pi}+aa^{\#}bb^{\#}aa^{\#}b(1-aa^{\#}bb^{\#})\\
&=&aa^{\#}bb^{\#}aa^{\#}b-aa^{\#}bb^{\#}aa^{\#}b[1-a^{\pi}]bb^{\#})\\
&=&0,\\
q(a(1+a^{\#}b))^{\tiny\textcircled{\#}}(1-q)&=&qaa^{\#}(a+b)^{\tiny\textcircled{\#}}aa^{\#}(1-q)\\
&=&qaa^{\#}(a+b)^{\tiny\textcircled{\#}}aa^{\#}(1-aa^{\#}bb^{\#})\\
&=&qaa^{\#}(a+b)^{\tiny\textcircled{\#}}aa^{\#}b^{\pi}\\
&=&qaa^{\#}(a+b)^{\tiny\textcircled{\#}}b^{\pi}-qaa^{\#}(a+b)^{\tiny\textcircled{\#}}a^{\pi}b^{\pi}\\
&=&aa^{\#}bb^{\#}aa^{\#}(a+b)^{\tiny\textcircled{\#}}b^{\pi}\\
&=&0.
\end{array}$$
We may write
$$a+aa^{\#}b=\left(
\begin{array}{cc}
c_1&0\\
c_2&c_3
\end{array}
\right)_q, (a+aa^{\#}b)^{^{\tiny\textcircled{\#}}}=\left(
\begin{array}{cc}
x_1&0\\
x_2&x_3
\end{array}
\right)_q.$$
As in the proof of Lemma 2.4, we prove that $c_1^{\tiny\textcircled{\#}}=x_1$.
Moreover, we have
$$\begin{array}{rll}
c_1&=&aa^{\#}bb^{\#}a[1+a^{\#}b]aa^{\#}bb^{\#}\\
&=&aa^{\#}bb^{\#}abb^{\#}+aa^{\#}bb^{\#}aa^{\#}baa^{\#}bb^{\#}\\
&=&aa^{\#}bb^{\#}a+aa^{\#}bb^{\#}aa^{\#}b\\
&=&aa^{\#}(1-b^{\pi})a+aa^{\#}bb^{\#}(1-a^{\pi})b\\
&=&aa^{\#}a-aa^{\#}b^{\pi}a+aa^{\#}bb^{\#}b-aa^{\#}bb^{\#}a^{\pi}b\\
&=&a-b^{\pi}a+aa^{\#}b\\
&=&aa^{\#}b+bb^{\#}a.
\end{array}$$ Therefore $aa^{\#}b+bb^{\#}a\in R^{\tiny\textcircled{\#}}$.

$(2)\Rightarrow (1)$ Let $q=aa^{\#}bb^{\#}$. Then $q^2=q\in R$.
Then $$\begin{array}{rll}
q&=&aa^{\#}bb^{\#}aa^{\#}bb^{\#}\\
&=&(1-a^{\pi})bb^{\#}aa^{\#}bb^{\#}\\
&=&bb^{\#}aa^{\#}bb^{\#}-a^{\pi}bb^{\#}aa^{\#}\\
&=&bb^{\#}aa^{\#}bb^{\#}.
\end{array}$$
Since $a,b\in R^{EP}$, we have $(aa^{\#})^*=aa^{\#},(bb^{\#})^*=bb^{\#}.$ Then $q^*=q$, i.e., $q\in R$ is a projection.
In view of Lemma 2.4, it will suffice to prove that $a+aa^{\#}b=a(1+a^{\#}a)\in R^{\tiny\textcircled{\#}}$.
We check that
$$\begin{array}{rll}
qa(1-q)&=&aa^{\#}bb^{\#}a[1-aa^{\#}bb^{\#}]\\
&=&aa^{\#}bb^{\#}a-aa^{\#}bb^{\#}a\\
&=&aa^{\#}(1-b^{\pi})a-aa^{\#}bb^{\#}a\\
&=&a-(1-a^{\pi})b^{\pi}a-aa^{\#}bb^{\#}a\\
&=&a-b^{\pi}a-aa^{\#}(1-b^{\pi})a\\
&=&-b^{\pi}a+(1-a^{\pi})b^{\pi}a\\
&=&0,\\
qaa^{\#}b(1-q)&=&aa^{\#}bb^{\#}aa^{\#}b[1-aa^{\#}bb^{\#}]\\
&=&aa^{\#}bb^{\#}[aa^{\#}b-aa^{\#}baa^{\#}bb^{\#}]\\
&=&aa^{\#}bb^{\#}[aa^{\#}b-aa^{\#}b(1-a^{\pi})bb^{\#}]\\
&=&aa^{\#}bb^{\#}[aa^{\#}b-aa^{\#}b^2b^{\#}-aa^{\#}ba^{\pi}bb^{\#}]\\
&=&0,\\
\end{array}$$
$$\begin{array}{rll}
(1-q)aa^{\#}bq&=&[1-aa^{\#}bb^{\#}]aa^{\#}baa^{\#}bb^{\#}\\
&=&[1-aa^{\#}bb^{\#}]aa^{\#}b\\
&=&aa^{\#}b-aa^{\#}bb^{\#}(1-a^{\pi})b\\
&=&aa^{\#}(1-b^{\pi})a^{\pi})b\\
&=&0,\\
(1-q)aa^{\#}b(1-q)&=&[1-aa^{\#}bb^{\#}]aa^{\#}b[1-aa^{\#}bb^{\#}]\\
&=&[aa^{\#}b-aa^{\#}bb^{\#}(1-a^{\pi})b][1-aa^{\#}bb^{\#}]\\
&=&[aa^{\#}bb^{\#}a^{\pi})b][1-aa^{\#}bb^{\#}]\\
&=&[aa^{\#}(1-b^{\pi})a^{\pi})b][1-aa^{\#}bb^{\#}]\\
&=&0.
\end{array}$$
Then $$a=\left(
\begin{array}{cc}
a_1&0\\
a_3&a_4
\end{array}
\right)_q, aa^{\#}b=\left(
\begin{array}{cc}
b_1&0\\
0&0
\end{array}
\right)_q.$$ Here $a_1=qaq, a_3=q^{\pi}aq, a_4=q^{\pi}aq^{\pi}$.
Then $$a+aa^{\#}b=\left(
\begin{array}{cc}
a_1+b_1&0\\
a_3&a_4
\end{array}
\right)_q,$$ where $$\begin{array}{rll}
a_4&=&(1-aa^{\#}bb^{\#})a(1-aa^{\#}bb^{\#})\\
&=&(1-aa^{\#}bb^{\#})ab^{\pi}\\
&=&ab^{\pi})-aa^{\#}bb^{\#}ab^{\pi}\\
&=&ab^{\pi}\in R^{\tiny\textcircled{\#}}.
\end{array}$$
Set $z=a+aa^{\#}b$. Then $zq^{\pi}=[a+aa^{\#}b][1-aa^{\#}bb^{\#}]=ab^{\pi}+aa^{\#}b-aa^{\#}baa^{\#}bb^{\#}=
ab^{\pi}+aa^{\#}b-aa^{\#}b(1-a^{\pi})bb^{\#}=ab^{\pi}\in R^{\tiny\textcircled{\#}}$. We verify that
$$\begin{array}{lll}
qzq&=&qaq+qaa^{\#}bq\\
&=&aa^{\#}bb^{\#}a+aa^{\#}bb^{\#}aa^{\#}b\\
&=&aa^{\#}(1-b^{\pi})a+aa^{\#}bb^{\#}(1-a^{\pi})b\\
&=&a-(1-a^{\pi})b^{\pi}a+aa^{\#}b-aa^{\#}(1-b^{\pi})a^{\pi}b\\
&=&a-b^{\pi}a+aa^{\#}b+aa^{\#}b^{\pi}a^{\pi}b\\
&=&aa^{\#}b+bb^{\#}a\in R^{\tiny\textcircled{\#}}.
\end{array}$$
Moreover, we have
$$\begin{array}{rll}
qzq^{\pi}&=&qa(1-q)+qaa^{\#}b(1-q)=0,\\
(zq^{\pi})^{\pi}q^{\pi}zq&=&q^{\pi}a_4^{\pi}a_3\\
&=&q^{\pi}(ab^{\pi})^{\pi}q^{\pi}aq\\
&=&q^{\pi}[1-aa^{\#}b^{\pi}][1-aa^{\#}bb^{\#}]abb^{\#}\\
&=&q^{\pi}[1-aa^{\#}b^{\pi}][a-aa^{\#}bb^{\#}a]\\
&=&q^{\pi}[a-[1-a^{\pi}]b^{\pi}a][1-a^{\#}bb^{\#}a]\\
&=&[1-a^{\#}bb^{\#}]bb^{\#}[1-aa^{\#}bb^{\#}]a\\
&=&[1-a^{\#}bb^{\#}][bb^{\#}-bb^{\#}aa^{\#}bb^{\#}]a\\
&=&[1-a^{\#}bb^{\#}][1-b^{\pi}]a^{\pi}bb^{\#}a\\
&=&[1-a^{\#}bb^{\#}]a^{\pi}[1-b^{\pi}]a\\
&=&0.
\end{array}$$ In light of Lemma 2.4, $a+aa^{\#}b=z\in R^{\tiny\textcircled{\#}}$.
Additionally, we have
$$\begin{array}{lll}
(a+aa^{\#}b)^{\tiny\textcircled{\#}}&=&\left(
\begin{array}{cc}
(a_1+b_1)^{\tiny\textcircled{\#}}&0\\
v&a_4^{\tiny\textcircled{\#}}
\end{array}
\right)_q,
\end{array}$$ where $v=-a_4^{\tiny\textcircled{\#}}a_3(a_1+b_1)^{\tiny\textcircled{\#}}.$

In light of Lemma 2.4, $a+b\in R^{\tiny\textcircled{\#}}$. By the argument above, we have $$\begin{array}{c}
aa^{\#}bb^{\#}(a+aa^{\#}b)^{\tiny\textcircled{\#}}(1-aa^{\#}bb^{\#})=0,\\
(a+aa^{\#}b)^{\tiny\textcircled{\#}}=aa^{\#}(a+b)^{\tiny\textcircled{\#}}aa^{\#}.
\end{array}$$ Therefore we get $$\begin{array}{rll}
ba^{\pi}(a+b)^{\tiny\textcircled{\#}}b^{\pi}&=&baa^{\#}(a+b)^{\tiny\textcircled{\#}}b^{\pi}\\
&=&baa^{\#}(a+b)^{\tiny\textcircled{\#}}(a^{\pi}+aa^{\#})b^{\pi}\\
&=&b[(bb^{\#}aa^{\#}][aa^{\#}(a+b)^{\tiny\textcircled{\#}}aa^{\#}]aa^{\#}b^{\pi}\\
&=&b[aa^{\#}-b^{\pi}aa^{\#}][aa^{\#}(a+b)^{\tiny\textcircled{\#}}aa^{\#}]aa^{\#}b^{\pi}\\
&=&baa^{\#}bb^{\#}aa^{\#}][aa^{\#}(a+b)^{\tiny\textcircled{\#}}aa^{\#}]aa^{\#}b^{\pi}\\
&=&b[aa^{\#}bb^{\#}](a+b)^{\tiny\textcircled{\#}}[aa^{\#}(1-aa^{\#}bb^{\#})]\\
&=&b[aa^{\#}bb^{\#}](a+b)^{\tiny\textcircled{\#}}[1-aa^{\#}bb^{\#}]\\
&=&0,
\end{array}$$ as asserted.\end{proof}

\begin{cor} Let $a,b\in R^{EP}$ and $a^{\pi}b,b^{\pi}a\in R^D$. If $aa^{\#}b=bb^{\#}a\in R^{\tiny\textcircled{\#}}$ and $\frac{1}{2}\in R$, then $a+b\in R^{\tiny\textcircled{\#}}$.\end{cor}
\begin{proof} Since $aa^{\#}b=bb^{\#}a\in R^{\tiny\textcircled{\#}}$ and $\frac{1}{2}\in R$, we have
$aa^{\#}b+bb^{\#}a=2aa^{\#}b\in R^{\tiny\textcircled{\#}}$. Also we have
$aba^{\pi}=aa^{\#}aba^{\pi}=abb^{\#}aa^{\pi}=0$. Similarly, $bab^{\pi}=bb^{\#}bab^{\pi}=baa^{\#}bb^{\pi}=0$.
Moreover, $a^{\pi}b^{\pi}a=a^{\pi}bb^{\#}a=a^{\pi}aa^{\#}b=0.$
Likewise, $b^{\pi}a^{\pi}b=0.$ This completes the proof by Theorem 3.1.\end{proof}

\begin{cor} Let $a,b\in R^{EP}$. If $ab=ba$ and $a^*b=ba^*$, then the following are equivalent:\end{cor}
\begin{enumerate}
\item [(1)] $a+b\in R^{\tiny\textcircled{\#}}.$
\vspace{-.5mm}
\item [(2)] $aa^{\#}b+bb^{\#}a\in R^{\tiny\textcircled{\#}}$.
\end{enumerate}
\begin{proof} Since $ab=ba$, by~\cite[Theorem 2.2]{D}, $a^{\#}b=ba^{\#}$. Hence, $a^{\pi}b=ba^{\pi}$, and so
$a^{\pi}b\in R^D$. Likewise, $b^{\pi}a\in R^D$. Moreover, we have $aba^{\pi}=bab^{\pi}=0$.

$(1)\Rightarrow (2)$ Since $ab=ba$ and $a^*b=ba^*$, we have $a^{\pi}(a+b)=(a+b)a^{\pi}$. Since $a\in R^{EP}$, we have $(a^{\pi})^*=a^{\pi}$ by ~\cite[Lemma 2.1]{XCB}. Hence $a^{\pi}(a+b)^*=(a+b)^*a^{\pi}$. In light of ~\cite[Theorem 3.2]{CZP}, we have $a^{\pi}(a+b)^{\tiny\textcircled{\#}}=(a+b)^{\tiny\textcircled{\#}}a^{\pi}$.
Accordingly, $a(a+b)^{\tiny\textcircled{\#}}a^{\pi}=0$. Likewise, $b(a+b)^{\tiny\textcircled{\#}}b^{\pi}=0$. Therefore
$aa^{\#}b+bb^{\#}a\in R^{\tiny\textcircled{\#}}$ by Theorem 3.1.

$(2)\Rightarrow (1)$ Clearly, $b^{\pi}a^{\pi}b=b^{\pi}(ba^{\pi})=0$. Likewise,
$b^{\pi}a^{\pi}b=0$. This completes the proof by Theorem 3.1.\end{proof}

\section{applications}

Let $M=\left(
  \begin{array}{cc}
    A&B\\
    C&D
  \end{array}
\right)\in {\Bbb C}^{(m+n)\times (m+n)},$ where $A\in {\Bbb C}^{m\times m},B\in {\Bbb C}^{m\times n},C\in {\Bbb C}^{n\times m},D\in {\Bbb C}^{n\times n}.$ We now applying the foregoing to present the conditions under which the block complex $M$ has core inverse.

\begin{lem} Let $A$ and $D$ have core inverses. If $D^{\pi}C=0$, then $\left(
  \begin{array}{cc}
    A & 0 \\
    C & D
  \end{array}
\right)$ has core inverse.\end{lem}
\begin{proof} This is obvious by ~\cite[Theorem 2.9]{XS}.\end{proof}

\begin{lem} Let $A,B\in {\Bbb C}^{n\times n}$ be EP and $AB=\lambda BA$ for some $\lambda \in {\Bbb C}^*$. If $AA^{\#}B+BB^{\#}A$ has core inverse, then $A+B$ has core inverse.\end{lem}
\begin{proof} Since every square complex matrix has Drazin inverse, $A^{\pi}B, B^{\pi}A\in ( {\Bbb C}^{n\times n})^D$. Since $AB=\lambda BA$ for some $\lambda \in {\Bbb C}^*$, we verify that $ABA^{\pi}=(\lambda BA)A^{\pi}=0$. Likewise, $BAB^{\pi}=0$. Since $AB=\lambda BA$, by virtue of ~\cite[Theorem 2.2]{D}, $A^DB=\lambda^{-1}BA^D$ and $AB^D=\lambda^{-1}B^DA$. Then $B^{\pi}A=A-B(B^DA)=A-\lambda (BA)B^D=A-ABB^D=AB^{\pi}$, and so $A^{\pi}B^{\pi}A=A^{\pi}AB^{\pi}=0$.
Likewise, $B^{\pi}A^{\pi}B=0$. Therefore $A+B\in {\Bbb C}^{n\times n}$ has core inverse by Theorem 3.1.\end{proof}

We are ready to prove:

\begin{thm} Let $A,D$ and $BC$ be EP. If $(BC)^{\pi}A=0, C(BC)^{\pi}=0,  (CB)^{\pi}D$ $=0, B(CB)^{\pi}=0, AB=\lambda BD$ and $DC=\lambda CA$ for some $\lambda \in {\Bbb C}^*$, then $M$ has core inverse.\end{thm}
\begin{proof} Write $M=P+Q$, where $$P=\left(
  \begin{array}{cc}
    A & 0 \\
    0 & D
  \end{array}
\right), Q=\left(
  \begin{array}{cc}
   0 & B \\
   C & 0
  \end{array}
\right).$$ Since $A$ and $D$ are EP, so is $P$.

Clearly, we have $$(Q^2)^D=\left(
  \begin{array}{cc}
   BC & 0 \\
   0 & CB
  \end{array}
\right)^D=\left(
  \begin{array}{cc}
   (BC)^D & 0 \\
   0 & (CB)^D
  \end{array}
\right).$$ Hence, $$Q^D=Q(Q^2)^D=\left(
  \begin{array}{cc}
  0& B(CB)^D\\
  C(BC)^D&0
  \end{array}
\right).$$ Thus, $$Q^{\pi}=\left(
  \begin{array}{cc}
  (BC)^{\pi}&0\\
  0&(CB)^{\pi}
  \end{array}
\right).$$
Hence, $QQ^D=Q^DQ, Q^D=Q^DQQ^D.$ Moreover, we have
$$\begin{array}{lll}
QQ^{\pi}&=&\left(
  \begin{array}{cc}
   0 & B \\
   C & 0
  \end{array}
\right)\left(
  \begin{array}{cc}
   (BC)^{\pi} & 0 \\
  0 & (CB)^{\pi}
  \end{array}
\right)\\
&=&\left(
  \begin{array}{cc}
  0&B(CB)^{\pi}\\
  C(BC)^{\pi} & 0
  \end{array}
\right)=0;
\end{array}$$ whence $Q=Q^2Q^D$. We infer that $Q$ has group inverse and $$QQ^{\#}=\left(
  \begin{array}{cc}
  BC(BC)^D&0\\
  0&CB(CB)^D
  \end{array}
\right).$$ Since $BC$ and $CB$ are EP, we see that $((BC)(BC)^{\#})^*=(BC)(BC)^{\#}$ and $((CB)(CB)^{\#})^*=(CB)(CB)^{\#}$.
This implies that $(QQ^{\#})^*=QQ^{\#}$, and so $Q$ has core inverse by Theorem 1.1.

It is easy to verify that
$$PQ=\left(
  \begin{array}{cc}
   0 & AB \\
    DC & 0
  \end{array}
\right)=\lambda\left(
  \begin{array}{cc}
  0 & BD \\
    CA & 0
  \end{array}
\right)=\lambda QP.$$

Moreover, we compute
$$\begin{array}{rll}
PP^{\#}Q&=&\left(
  \begin{array}{cc}
    AA^{\#} & 0 \\
    0 & DD^{\#}
  \end{array}
\right)\left(
  \begin{array}{cc}
   0 & B \\
   C & 0
  \end{array}
\right)\\
&=&\left(
  \begin{array}{cc}
   0 & 0 \\
   DD^{\#}C & 0
  \end{array}
\right),\\
QQ^{\#}P&=&\left(
  \begin{array}{cc}
  BC(BC)^{\#}&0\\
  0&CB(CB)^{\#}
  \end{array}
\right)\left(
  \begin{array}{cc}
    A & 0 \\
    0 & D
  \end{array}
\right)\\
&=&\left(
  \begin{array}{cc}
     BC(BC)^{\#}A & 0 \\
    0 & CB(CB)^{\#}D
  \end{array}
\right).
\end{array}$$
Thus, $$\begin{array}{rll}
PP^{\#}Q+QQ^{\#}P&=&\left(
  \begin{array}{cc}
   BC(BC)^{\#}A& 0 \\
   DD^{\#}C & CB(CB)^{\#}D
  \end{array}
\right)\\
&=&\left(
  \begin{array}{cc}
   (I_m-(BC)^{\pi})A& 0 \\
   DD^{\#}C &(I_n-(CB)^{\pi})D
  \end{array}
\right)\\
&=&\left(
  \begin{array}{cc}
   A& 0 \\
   DD^{\#}C&D
  \end{array}
\right).
\end{array}$$ Clearly, $D^{\pi}(DD^{\#}C)=0$. Since $A$ and $D$ have core inverses, by virtue of Lemma 4.1, $\left(
  \begin{array}{cc}
   A&0\\
   DD^{\#}C&D
  \end{array}
\right)$ has core inverse. Therefore $PP^{\#}Q+QQ^{\#}P$ has core inverse. By using Lemma 4.2, $M=P+Q$ has core inverse, hence the result.
\end{proof}

\begin{cor} Let $A,D$ be EP and $BC$ be invertible. If $AB=\lambda BD$ and $DC=\lambda CA$ for some $\lambda \in {\Bbb C}^*$, then $M$ has core inverse.\end{cor}
\begin{proof} Since $BC$ be invertible, then so is $CB$. Hence, $(BC)^{\pi}=0$ and $(CB)^{\pi}=0$. This completes the proof by
Theorem 4.3.\end{proof}

\end{document}